\documentclass[aap,MSNbibl,citesort,seceqn,dvips]{arximspdf}
\usepackage{graphicx}

\doi{10.1214/12-AAP840} %kopijuoti is PTS
\volume{23}
\issue{1}
\pubyear{2013}
\firstpage{283}
\lastpage{307}

\makeatletter
\newcommand{\eqref}[1]{(\ref{#1})}
\newtheorem{theorem}{Theorem}[section]
\newtheorem{lemma}[theorem]{Lemma}
\newtheorem{proposition}[theorem]{Proposition}
\newtheorem{corollary}[theorem]{Corollary}
\newtheorem{conj}[theorem]{Conjecture}
\newproclaim{remark}{Remark}
\newproclaim{fact}[theorem]{Fact}
\newcommand{\bbN}{\mathbb{N}} %{1,2,...}

\newcommand{\E}{\mathbb{E}}
\renewcommand{\P}{\mathbb{P}}

\newcommand{\eps}{\varepsilon}

\newcommand{\gem}{\operatorname{GEM}}
\newcommand{\pd}{\operatorname{PD}}

\newcommand{\lam}{\lambda}
\renewcommand{\b}{\beta}
\newcommand{\order}{\asymp} % same order of magnitude as

\newcommand{\HH}{\mathcal{H}}
\newcommand{\FF}{\mathcal{F}}
\newcommand{\iii}{\mathcal{I}}
\newcommand{\WW}{\mathcal{W}}
\newcommand{\XX}{\mathcal{X}}
\newcommand{\bW}{\mathbf{W}}
\newcommand{\bv}{\mathbf{v}}

\makeatother

\begin{document}
\begin{frontmatter}

\title{Poisson--Dirichlet branching random walks}
\runtitle{Poisson--Dirichlet branching random walks}

\begin{aug}
\author[A]{\fnms{Louigi} \snm{Addario-Berry}\corref{}\ead[label=e1]{louigi@math.mcgill.ca}\thanksref{t1}}
\and
\author[B]{\fnms{Kevin} \snm{Ford}\ead[label=e2]{ford@math.uiuc.edu}\thanksref{t2}}
\thankstext{t1}{Supported by an NSERC Discovery Grant.}
\thankstext{t2}{Supported by NSF Grant DMS-09-01339.
The research was conducted in part while the K. Ford
was visiting the Institute for Advanced Study, supported by grants from the
Ellentuck Fund and The Friends of the Institute For Advanced Study.}
\runauthor{L. Addario-Berry and K. Ford}
\affiliation{McGill University and University of Illinois at Urbana-Champaign}
\address[A]{McGill University\\
1005-805 Sherbrooke West\\
Montreal, Quebec\\
H3A 2K6, Canada\\
\printead{e1}} %adresu isvedimo komanda gale!
\address[B]{Department of Mathematics\\
University of Illinois at Urbana-Champaign\\
1409 West Green St., Urbana\\
Illinois 61801, USA\\
\printead{e2}}
\end{aug}

% HISTORY:
\received{\smonth{2} \syear{2011}}
\revised{\smonth{8} \syear{2011}}

% ABSTRACT
%
\begin{abstract}
We determine, to within $O(1)$, the expected minimal position at level $n$
in certain branching random walks. The walks under consideration
have displacement vector $(v_1,v_2,\ldots)$, where each $v_j$ is the sum
of $j$ independent $\operatorname{Exponential}(1)$ random variables
and the different
$v_i$ need not be independent.
In particular, our analysis applies to the Poisson--Dirichlet branching
random walk and to the Poisson-weighted infinite tree.
As a corollary, we also determine the expected height of a random
recursive tree to within $O(1)$.
\end{abstract}

% KEYWORDS
%

\end{frontmatter}

%s1 ###
\section{Introduction}\label{S:intro}
A branching random walk starts from an initial particle, the \textit{root}, with position $0$. The
root produces some number of children, who are randomly displaced from
their parent
according to some displacement law.
Each child in turn produces some number of children, who are displaced
from the position of their parent according to the same law; and so on.
In general, the displacements of siblings relative to their parent may
be dependent, but for distinct
particles $v$ and $w$, the displacements of the children of $v$ and of
the children of $w$ must be independent.
When the displacements are nonnegative, this is often called an age-dependent
branching process, and the displacements are thought of as ``times to birth.''

There is a natural tree associated with a branching random walk,
where the vertices correspond to particles, and an edge from
parent to child is weighted with the child's displacement from its parent.
More precisely, let $T$ be the \textit{Ulam--Harris} tree,\vspace*{1pt}
which has vertex set $V = \bigcup_{n=0}^{\infty} \bbN^n$ (we think
of elements of $\bbN^n$ as concatenations of $n$ integers, and take
$\bbN^0=\{\varnothing\}$),
is rooted at $\varnothing$, and has an edge from $v$ to $vi$ for each $v
\in V$ and each $i \in\bbN$.
We call $\bbN^n$ the $n$th generation of $T$, and for $v=v_1,\ldots, v_n
\in\bbN^n$, we say that $v$ has \textit{parent}
$p(v)=v_1,\ldots, v_{n-1}$ and \textit{children} $vi$, $i \in\bbN$. (We
will usually write $T_n$ in place of $\bbN^n$ for readability.)

Now suppose $\mathbf{X}=(X_i\dvtx i \in\bbN)$ is a random vector, where
each $X_i \in\mathbb{R}\cup\{+\infty\}$.
We do \textit{not} require that the entries of $\mathbf{X}$ are
independent of one another---this will be important below.
Then we form a branching random walk by marking each vertex $v \in V$
with an independent copy $\mathbf{X}^v=(X^v_i\dvtx i \in\bbN)$ of
$\mathbf{X}$. Write $\mathcal{T}$ for the pair $(T,\{\mathbf
{X}^v\dvtx v
\in
V\})$; then $\mathcal{T}$ is our branching random walk. We call
$\mathbf
{X}$ the
\textit{displacement vector} of $\mathcal{T}$.
\setcounter{footnote}{2}\footnote{For the formal details of a
probabilistic construction
of branching random walks, see, for example,~\cite
{harris63branching}.}
For each $v \in V$ and $i \in\bbN$, we regard $X^v_i$ as the
displacement from $v$ to $vi$, and let
$S(v)=S(v,\mathcal{T})$ be the sum of the displacements on the path
from the root to $v$ [formally, if $v=v_1,\ldots, v_n$, then $S(v) = \sum
_{i=1}^n X^{p(v_1,\ldots, v_i)}_{v_i}$, and this sum is taken to be
$+\infty$ if any of its elements are $+\infty$]. We say $\mathcal{T}$
has \textit{finite branching} if almost surely all but finitely many
coordinates of $\mathbf{X}$ are equal to $+\infty$.

For $n\in\bbN$, let $M_n=\inf(S(v)\dvtx v \in\bbN^n)$. In all situations
we consider in this paper, this infimum is attained, so $M_n$ is the
minimal displacement of any individual in the $n$th generation. The
minimal displacement is one of the most well-studied parameters
associated with branching random walks.
It has been known since the 1970s~\cite{hammersly74postulates,kingman75first,biggins76first}
that under quite
general conditions, $M_n$
grows asymptotically linearly with lower-order corrections. Recently
there have been substantial developments in understanding the finer
behavior of $M_n$ on two fronts: first, convergence results for the
lower order corrections~\cite{addario07brw,aidekon10weak,hu2009}; and
second, the concentration of $M_n$ about its mean (or median)
\cite{addario07brw,bramson06tightness,chauvin05random}. We refer to these as
the global behavior and the local behavior of $M_n$, respectively.
Under suitable conditions, $M_n$ generally seems to exhibit the
following behavior: for some constants $\alpha\in\mathbb{R}$ and
$\beta> 0$,
$\operatorname{median}(M_n) = \alpha n+\beta\log n + O(1)$, and,
furthermore, $M_n/ n \to\alpha$ almost surely and $(M_n - \alpha
n)/\log n \to\beta$ in probability
(but \textit{not} almost surely~\cite{hu2009}). Also, under sufficiently
strong moment conditions for the displacements,
$\E\{ \exp(\gamma|M_n-\E{M_n}|) \}<\infty$ for some $\gamma> 0$
and all
$n$. (In fact, in some cases the upper tail of $M_n-\E{M_n}$ is even
known to decay doubly-exponentially quickly~\cite{bachman_brw,ford09pratt}.)

To date, however, all the results of the kind described in the
preceding paragraph that we are aware of
require that the branching random walk has finite branching. In this
paper we study the global behavior of $M_n$ for a class of branching
random walks which
\textit{do not} have finite branching. The class we consider is rather
restricted but nonetheless contains at least two interesting special
cases, one related to the factorization of random integers, and one
related to the analysis of algorithms. Say that $\mathbf{X}$ has
\textit{exponential steps} if for all $i$, $X_i$ is distributed as the
sum of
$i$ independent $\operatorname{Exponential}(1)$ random variables. The
main result of
this paper is the following theorem. For short, we denote
\newcommand{\tM}{\widetilde{M}}
\newcommand{\er}{e}
\[
\tM_n = \operatorname{median}(M_n) := \sup\bigl\{x\dvtx\P\{M_n < x\} <
1/2\bigr\}.
\]

\begin{theorem} \label{T1}
If $\mathbf{X}$ has exponential steps, then
\[
\tM_n = \frac{n}{e}+\frac{3}{2e}\log n + O(1).
\]
\end{theorem}

\begin{remark}
The $O(1)$ term is uniform over $n$ and over all
BRW for
which $\mathbf{X}$ has exponential steps.
\end{remark}

\begin{remark}
Independently of the current work, \'Elie
A\"id\'ekon~\cite{aidekonpaper} has recently proved, for a quite general
family of random walks (including those considered in this paper), that
$M_n - \tM_n$ converges in distribution to a random variable $M^*$, and
describes the distribution of $M^*$ in terms of a functional of the
limit of the derivative martingale associated to the branching random walk.
\end{remark}

Using methods from~\cite{ford09pratt}, we can deduce from Theorem~\ref{T1}
uniform exponential tails for $M_n$.
In the next theorem and at other points throughout the paper, we will
use the Vinogradov notation $f \ll g$ which means $f=O(g)$, with
subscripts indicating
dependence on any parameter, for example, $f\ll_k g$ means the constant
implied by
the $\ll$ symbol may depend on $k$ but not on any other variable.

\begin{theorem}\label{TFord}
If $\mathbf{X}$ has exponential steps, then
for any $c_1<\er$, we have
\[
\P\{M_n \le\tM_n -x\} \ll_{c_1} \er^{-c_1 x}\qquad (n\ge1, x\ge0)
\]
and for any $c_2 < 1$,
\[
\P\{ M_n \ge\tM_n + x \} \ll_{c_2} \er^{-c_2 x}\qquad
(n\ge1, x\ge0).
\]
\end{theorem}

Again, the above estimates are uniform over all BRW under
consideration. Also, Theorem~\ref{TFord} implies that $\tM_n = \E
{M_n}+O(1)$, and so both Theorems~\ref{T1} and~\ref{TFord} hold with
$\tM_n$ replaced by $\E{M_n}$.

The simplest example of a displacement vector with exponential steps is
obtained by taking $\mathbf{X}=(E_1,E_1+E_2,\ldots)$,
where $\{E_i\}_{i \in\bbN}$ are i.i.d.\break $\operatorname
{Exponential}(1)$ random variables.
In this case $\mathcal{T}$ is called the \textit{Poisson-weighted infinite
tree}~\cite{aldous2004omp}
and has been used very effectively in probabilistic combinatorial
optimization. It also arises in the analysis of an important tree-based
data structure in the following way.
Order the elements of $\mathcal{T}$ in increasing order of displacement
as $\{w_i\}_{i \in\bbN}$, so, in particular, we have $w_1=\varnothing,
w_2=1 \in\bbN^1$,
and either $w_3=2 \in\bbN^1$ or $w_3=11 \in\bbN^2$. Now for each $m$
let $Z_m$ be the subtree of $\mathcal{T}$ induced by $w_1,\ldots,w_{m}$.
By the memoryless property of the exponential, it follows that the
parent of $w_{m+1}$ is a uniformly random element of $Z_m$---in other words, $Z_m$ is a
\textit{random recursive tree} for all $m$.
This connection is well known~\cite{pittel94recursive}.

$Z_m$ is also the subtree of $\mathcal{T}$ induced by the set of nodes
of displacement at most $S(w_m)$. [Also, it is straightforwardly shown by
induction and the memoryless property of the exponential that the
families $(Z_m)_{m \in\bbN}$ and $(S(w_m))_{m \in\bbN}$ are independent,
but we will not need this.] Let $H_m$ be the \textit{height} of
$Z_m$---the largest generation containing a node of $Z_m$. In other words,
$H_m=\max\{n\dvtx M_n \leq S(w_m)\}$, which is the representation
that will
be useful below.
Devroye~\cite{devroye87branching} showed that $H_m/\log m \to e$ almost
surely and in expectation, and Pittel~\cite{pittel94recursive} provided
a different proof of the almost sure convergence. As a straightforward
consequence of Theorems~\ref{T1} and~\ref{TFord}, we obtain the
following more precise information.
\begin{corollary}\label{cor1.3}
The height $H_m$ of a random recursive tree on $m$ nodes satisfies
$\E{H_m} = e\log m - \frac{3}{2}\log\log m + O(1)$. Furthermore, for
all $c'<\frac1{2e}$, all $m \geq1$, $k \geq1$,
\[
\P\{ |H_m-\E{H_m}| \geq k \} \ll_{c'} e^{-c' k}.
\]
\end{corollary}

Since the proof of this corollary is very short, we include it in the
\hyperref[S:intro]{Introduc-} \hyperref[S:intro]{tion}.
In the proof we write $\operatorname{har}(s) = \sum_{i=1}^s 1/i$.
\begin{pf*}{Proof of Corollary \protect\ref{cor1.3}}
The random variable $S(w_m)$ is distributed as the sum, $F_1+\cdots
+F_{m-1}$, of
independent random variables with $F_i$ having $\operatorname
{Exponential}(i)$ distribution for $i=1,\ldots,m-1$. Equivalently,
$S(w_m)$ is distributed as the maximum of $m-1$ i.i.d. $\operatorname
{Exponential}(1)$
random variables. Thus,
$\E{S(w_m)}=\operatorname{har}(m-1)$ and for all $x > 0$,
\begin{eqnarray}\qquad
\P\{ S(w_m) \geq\operatorname{har}(m-1)+x \} & \leq&(m-1)
e^{-(\operatorname
{har}(m-1)+x)} \leq e^{-x}, \label{eq:max_upper}\\
\P\{ S(w_m) \leq\operatorname{har}(m-1)-x \} & =& \bigl(
1-e^{-(\operatorname {har}(m-1)-x)} \bigr)^{m-1}
\leq e^{-e^{x-1}}.\label{eq:max_lower}
\end{eqnarray}
Now write
\[
d(m)=\max\{n\dvtx \tM_n \leq\operatorname{har}(m-1)\} = e\log m -
\tfrac{3}{2}
\log\log m + O(1)
\]
and note that $\tM_{d(m)}=\operatorname{har}(m-1)+O(1)$ by Theorem
\ref{T1}.
It follows that for $k\geq1$, if $H_m \geq d(m)+k$, then either
\[
M_{d(m)+k} \leq\operatorname{har}(m-1)+\frac{k}{2e} \leq\tM
_{d(m)+k}-\frac
{k}{2e}+O(1),
\]
or
\[
S(w_m) \geq\operatorname{har}(m-1)+\frac{k}{2e}.
\]
By Theorem~\ref{TFord} and (\ref{eq:max_upper}), it follows that $\P
\{ H_m \geq d(m)+k \} \ll_{c_1} e^{-c_1 k/(2\er) }$ for each
$c_1 <\er$. A similar argument using Theorem~\ref{TFord} and (\ref
{eq:max_lower}) shows the bound $\P\{ H_m \leq d(m)-k \} \ll_{c_2} e^{-c_2
k/(2\er)}$ for each $c_2 < 1$.
\end{pf*}

Another important example of a displacement vector with exponential
steps arises
from a discrete time random fragmentation process.
%The \textit{Griffiths--Engen--McCloskey} $\gem$ distribution is a
%distribution on
%the space of sequences of nonnegative real numbers of sum $1$, which
%for our %purposes is most
%usefully defined as follows.
Let $U_1,U_2,\ldots$ be independent uniform $[0,1]$ random variables.
Set $G_1=U_1$ and for $i > 1$ set $G_i = (1-U_1)\cdot\,\cdots\,\cdot
(1-U_{i-1})U_i$. The distribution of the sequence
\[
{\mathbf G} = (G_1,G_2,\ldots)
\]
was first studied, in greater generality, in~\cite{halmos_alms}. (One
motivation for Halmos' paper was a problem about
loss of energy of neutrons after many collisions; after each collision
the neutron loses a random fraction of its current energy.)
${\mathbf G}$~is also a special case of the
\textit{Griffiths--Engen--McCloskey} $\gem$ distribution.
Further, $(G_{\sigma(1)},G_{\sigma(2)},\ldots)$ has the
Poisson--Dirichlet (or $\pd$) distribution,
where $\sigma\dvtx\bbN\to\bbN$ is the permutation that arranges
the terms
of $(G_1,G_2,\ldots)$ in decreasing order.
(We remark that both the $\gem$ and the $\pd$ distributions as defined
above are in fact special cases from a more general two-parameter
family of distributions~\cite{pitman2002csp}---in the standard notation, we are considering
the $\gem(0,1)$ and $\pd(0,1)$ distributions.)
The PD distribution arises in a number of natural decomposition
situations, such
as factorization of large random integers \cite
{Bil72,donnelly1993asymptotic} and cycle lengths of
random permutations~\cite{pitman2002csp}.

Letting $X_k=-\log G_k$ for each $k$ yields a vector $(X_1,X_2,\ldots
)$ with
exponential steps. We refer to the resulting branching random walk as a
\textit{Poisson--Dirichlet branching random walk}.
This example has more complicated dependence between the $X_i$ than the
first example.
Since $\sum_{i=1}^{\infty}G_i=1$ almost surely, there is another way to
think of the branching random walk.
Imagine that an object of mass $m$ is placed at the root $\varnothing$.
The root divides this mass
into pieces according to the vector ${\mathbf G}^{\varnothing}$ and sends
the pieces to its children, sending a mass $mG_k^{\varnothing}$ to
its $k$th child. This rule is repeated recursively, so each node $v$
sends proportion $G_k^{v}$ of the mass it receives to its $k$th child $vk$.
This structure is variously called a \textit{multiplicative cascade} or,
more commonly at the moment, a \textit{fragmentation process} \cite
{bertoin2006rfa}.
The special case of Theorem~\ref{T1} when $\mathcal{T}$ is a
Poisson--Dirichlet branching random walk is used in~\cite{ford09pratt}
to analyze a
tree model related to primality testing, proving heuristic evidence for
the behavior of the distribution of tree heights. In this special case of
a PD branching random walk, a much stronger estimate for the right tail of
$M_n$ was proved in~\cite{ford09pratt}, namely, for any $c_3<1$,
\[
\P\{ M_n \ge\tM_n + x \} \le\exp\{ - e^{c_3 x - c_4} \}\qquad
(n\ge1, x\ge0),
\]
where $c_4$ is a constant depending on $c_3$. Such a right tail bound
cannot hold in general; for example,
for the case of $\mathcal{T}$ being a Poisson-weighted infinite tree,
we have $\P\{ M_1\ge x\} = e^{-x}$.
(It seems likely that among branching random walks with exponential
steps, the Poisson-weighted infinite tree and the Poisson--Dirichlet
branching random walk are extremal examples, with the former having the
heaviest tails for $M_n-\tM_n$ and the latter the strongest tail bounds
for $M_n-\tM_n$. However, we do not have a precise conjecture in this
direction.)

The \textit{Pratt tree} for a prime $p$ has root $p$ whose children are
the prime factors of $p-1$; the subtrees of the children of the
root are recursively constructed in the same fashion (stopping when
$p=2$). We let $H(p)$ be the height of the Pratt tree for~$p$.
It is easily seen that the height is always at most $(\log p)/(\log2)
+ 1$.
Such trees were used by Pratt~\cite{pratt75prime} to show that
if $p$ is prime, then there exists a certificate (formal proof) of the
primality of $p$, of length $O(H(p)\log p)=O((\log p)^2)$.
It is then of interest to understand the ``typical'' behavior of
$H(p)$.~\cite{ford09pratt} uses Theorems~\ref{T1} and~\ref{TFord} to
support the following conjecture.

\begin{conj}[(\cite{ford09pratt}, Conjecture 3)]
There exist constants $c,c'>0$ and real numbers $\{E(p)\dvtx p
\operatorname
{prime}\}$ such that
\begin{itemize}
\item$H(p)= e\log p - \frac{3}{2} \log\log p + E(p)$,
\item for all $z \geq0$, and $x \geq0$,
\[
e^{-c' z} \pi(x) \ll|\{\operatorname{primes} p \leq x\dvtx E(p)
\geq z\}| \ll
e^{-cz} \pi(x)
\]
and
\[
|\{\operatorname{primes} p \leq x\dvtx E(p) \leq-z\}| \ll\exp
(-e^{cx})\pi(x).
\]
\end{itemize}
Here $\pi(x)$ is the number of primes which are at most $x$.
\end{conj}

The structure of the remainder of the paper is as follows. In Section~\ref{sec2}
we introduce a little additional notation. In Section~\ref{sec3} we use
straightforward calculations to prove weak bounds on the likely value
of $M_n$, and to ``reduce the search space'' of nodes in $T_n$ which
have a chance of attaining the minimal displacement $M_n$. Section~\ref{rsrw}
studies the sample path properties of a uniformly random element of
certain ``homogeneous'' subsets of $T_n$, and forms a key step of the
proof. In Section~\ref{sec5} we prove the lower bound of Theorem~\ref{T1}, and
in Section~\ref{sec6} we prove the upper bound. Finally, the details of the
proof of Theorem~\ref{TFord} are found in
Section~\ref{sec:tails}.

%s2 ###
\section{Notation}\label{sec2}

Given $v=v_1v_2,\ldots, v_n \in V$, we let $h(v) = \sum_{i=1}^n v_i$, and
remark that $S(v)$ has distribution
Gamma$(h(v))$. If $v \in T_n$, we write $k(v)=h(v)-n$, and write
$T_{n,k}$ for the set of nodes $v \in T_n$ with $k(v)=k$.
We denote by $T_n(x)$ [resp., $T_{n,k}(x)$] the set of nodes of $T_n$
(resp., $T_{n,k}$) with displacement at most $x$.\vadjust{\goodbreak}

The Bachmann--Landau notation $o(\cdot)$ and $O(\cdot)$ have their usual meaning.
As mentioned earlier, we use the Vinogradov notation $f\ll g$ which
means $f=O(g)$.
We also use the Hardy notation $f \asymp g$ which means $f=O(g)$ and $g=O(f)$.
Constants implied by these symbols are absolute unless otherwise
indicated, for example, by
a subscript.

%%%%%%%%%%%%%%%%%%%%%%%%%%%%%%%%%%%%%%%%%%%%%%%%%%%%%%%%%%
%
%s3 ###
\section{Some basic expectations}\label{sec3}
%
%%%%%%%%%%%%%%%%%%%%%%%%%%%%%%%%%%%%%%%%%%%%%%%%%%%%%%%%%%

In order to restrict the set of nodes, we need to consider when
searching for the precise location of $M_n$, we first
assert the following two straightforward facts, whose proofs are forthcoming.

\begin{lemma}\label{lemma_ab}
\textup{(a)} The expected number of nodes $v \in T_n$ with $|h(v) - (1+1/e)n|
\le\sqrt{n}$ and with $S(v) \leq n/e+\log n/(2e)$ is $\gg1$.\vspace*{-6pt}
\begin{longlist}
\item[(b)] The expected number of nodes $v \in T_n$ with $S(v) \leq n/e +
(2/e)\log n$ and with $|h(v)-(1+1/e)n| > \sqrt{6n \log n}$ is $O(n^{-1/2})$.
\end{longlist}
\end{lemma}

Together, (a) and (b) suggest that in order to find $M_n$, it should
suffice to look at nodes in $T_n$ satisfying $h(v)=(1+1/e)n+O(\sqrt
{n})$, as will indeed be the
case. In proving (a) and (b), we will in fact prove more general bounds
that will be useful throughout the paper.

We first remark that for $v \in V$ with $h(v)=h$, $S(v)$ has density function
\[
\gamma_h(x)= \frac{x^{h-1} e^{-x}}{(h-1)!}\qquad
(x\ge0). %= \frac{(1+o_h(1))}{\sqrt{2\pi(h-1)}}
\]
For all $n \geq1$, $k \geq0$, we have
\begin{equation}\label{eq:tnk}
|T_{n,k}| = \pmatrix{{n+k-1}\vspace*{2pt}\cr{k}},
\end{equation}
so the sum of the density functions for nodes $v \in T_{n,k}$ is
\[
f_{n,k}(x) = \pmatrix{{n+k-1}\vspace*{2pt}\cr{k}} \gamma_{n+k}(x) =
\frac{x^{n+k-1}e^{-x}}{k!(n-1)!} = \frac{x^k}{k!} \cdot\frac
{x^{n-1}e^{-x}}{(n-1)!}.
\]
This function will play a significant role, and we now derive bounds on
its value for a variety of ranges of $k$ and $x$.
We remark that assertions (a) and (b), above, state, in particular,
that to find $M_n$ we should take both $k$ and $x$ near $n/e$. Thus,
writing $k=(n+r)/e$ and $x=(n+y)/e$,
by Stirling's formula, we have
\begin{eqnarray}\label{eq:fnkapprox}
f_{n,k}(x) &=& \frac{(1+O({1}/{n}+{1}/{k}))}{n+y} \sqrt{\frac
{n}{n+r}} e^{(r-y)/e} \biggl( 1-\frac{r-y}{n+r} \biggr)^{(n+r)/e}
\nonumber
\\[-8pt]
\\[-8pt]
\nonumber
&&{}\times  \biggl( 1+\frac y n
\biggr)^{n} \frac{e^{3/2}}{2\pi}.
\end{eqnarray}
When $r=O(\sqrt{n})$, $y=O(\sqrt{n})$, we have $(1+y/n)^{n} \asymp
e^{y}$ and
\[
\bigl(1-(r-y)/(n+r)\bigr)^{(n+r)/e} \asymp e^{-(r-y)/e}
\]
and so obtain the simpler approximation
\[
f_{n,k}(x) \asymp\frac{e^{y}}{n}.
\]
Consequently,
\begin{equation} \label{eq:efnkapprox}
\E{\biggl|\biggl\{v \in T_{n,k} \dvtx S(v) \leq\biggl(n+\frac12\log n\biggr)\big/e\biggr\}\biggr|} \asymp
\int
_{0}^{(\log n)/2} \frac{e^y}{n} \asymp n^{-1/2}
\end{equation}
for any fixed $k=n/e+O(\sqrt{n})$---where the constants implicit in
$O(\sqrt{n})$ and in~(\ref{eq:efnkapprox})
may depend on each other---and so we obtain
\[
\E\biggl|\biggl\{v \in T_{n,k} \dvtx S(v) \leq\biggl(n+\frac12\log n\biggr)\Big/e,|k-n/e|\leq
\sqrt
{n}\biggr\}\biggr|
\gg1.
\]
This justifies claim (a) of Lemma~\ref{lemma_ab}, and we now turn to
Lemma~\ref{lemma_ab}(b).
The next lemma is \cite[Lemma 5.1]{ford09pratt}, and we give a
different proof below.

\begin{lemma}\label{ETnx}
For all $n$ and $x\ge0$,
\[
\E|T_{n}(x)| = \frac{x^n}{n!}.
\]
\end{lemma}

\begin{pf}
We have
\[
\E|T_n(x)| =\sum_{k\ge0} \sum_{v\in T_{n,k}} \P\{ S(v)\le x \} =
\sum_{k\ge0} \int_0^x f_{n,k}(t) \,dt = \frac{x^n}{n!}.
\]
\upqed\end{pf}

It follows immediately from Lemma~\ref{ETnx} and Stirling's formula
that the median of $M_n$ is $\ge\frac{n}{e}+\frac{1}{2e}\log n + O(1)$.

We next obtain bounds on the probability that $k$ is very different
from $x$ when $x \geq n/(2e)$. First we quote easy bounds for the tails
of the Poisson
distribution.

\begin{proposition}\label{poisson}
If $z>0$ and $0<\alpha\le1\le\b$, then
\[
\sum_{k\le\alpha z} \frac{z^k}{k!}
< \biggl(\frac{e}{\alpha}\biggr)^{\alpha z},\qquad
\sum_{k\ge\b z} \frac{z^k}{k!}
< \biggl(\frac{e}{\beta}\biggr)^{\beta z}.
\]
\end{proposition}

\begin{pf} We have
\[
\sum_{k\le\alpha z} \frac{z^k}{k!} =
\sum_{k\le\alpha z} \frac{(\alpha z)^k}{k!}\biggl (\frac{1}{\alpha}\biggr)^k
\le
\biggl(\frac{1}{\alpha}\biggr)^{\alpha z} \sum_{k \le\alpha z} \frac{(\alpha z)^k}{k!}
< \biggl(\frac{e}{\alpha}\biggr)^{\alpha z}.
\]
The second inequality follows in the same way.\vadjust{\goodbreak}
\end{pf}

An easy corollary is the following.

\begin{lemma}
For $0\le t\le x^{1/6}$,
\[
\sum_{\{k\dvtx |k-x| \geq t \sqrt{x} \}} f_{n,k}(x) \ll
e^{-t^2/2} \frac{x^{n-1}}{(n-1)!}.
\]
\end{lemma}

Taking $t=\lceil\sqrt{5\log n}\rceil$ and integrating the above
bound over
$n/e \leq x \leq n/e+ (2/e)\log n$,
we obtain the bound
\[
\E\biggl|\biggl\{ v \in T_n\dvtx\frac{n}{e} \leq S(v) \leq\frac{n+2\log
n}{e}, |h(v)-S(v)| \geq\sqrt{5n\log n} \biggr\}\biggr| =O\biggl(\frac{1}{n^{1/2}}\biggr).
\]
Since $\sqrt{5n\log n} + (2/e)\log n < \sqrt{6n\log n}$ for $n$ large,
combining the preceding expectation bound with
Lemma~\ref{ETnx} (applied with $x=n/e$) and Stirling's formula, it
follows that
\[
\E\biggl\{ \biggl|\bigcup_{\{k\dvtx|k-n/e| \geq\sqrt{6 n \log n}\}}
T_{n,k}\bigl((n+2\log n)/e\bigr)\biggr| \biggr\} =O\biggl(\frac{1}{n^{1/2}}\biggr),
\]
which establishes Lemma~\ref{lemma_ab}(b).

%We remark that
%$f_{n,0}(x) = x^{n-1}e^{-x}/n! \leq$
%Summing $f_{n,k}(x)$ over $k \geq1$ and using the bound (
%uniformly over $x$ bounded away from zero

%%%%%%%%%%%%%%%%%%%%%%%%%%%%%%%%%%%%%%%%%55
%
%s4 ###
\section{Randomly sampled random walk}\label{rsrw}
%
%%%%%%%%%%%%%%%%%%%%%%%%%%%%%%%%%%%%%%%%%%%%

%The two preceding issues are of different kinds. The first issue
%concerns the \textit{number} of displacements along a path, and their
%positions along that path. The
%second issue concerns the \textit{values} of these displacements (which
%we refer to as the shape of the path). We will address these issues by
%restricting our attention to a subset of $T_n$ for which neither
%kind of problem can occur; the current section is devoted to defining
%this special subset. To do so, we must first introduce some notation.

For integers $n \geq1$, $k \geq0$ and a vertex $v=v_1,\ldots,v_n\in
T_{n,k}$, let $h_i(v)=h(v_1,\ldots, v_i)$ and
$W_i(v)=S(v_1,\ldots, v_i)$ for $1\le i\le n$, and write $\bW
(v)=(W_1(v),\ldots,W_n(v))$. We write $\bW$, $W_i$ and $h_i$ in place
of $\bW(v)$, $W_i(v)$ and $h_i(v)$ when
$v$ is clear from context.
We will always write $\bv_{n,k}$ for a uniformly random element of
$T_{n,k}$, independent of
$\bv_{n',k'}$ for $(n,k) \neq(n',k')$, and write
$\WW_{n,k}$ for the distribution of the sequence $\WW(\bv
_{n,k})=(W_1(\bv_{n,k}),\ldots,W_n(\bv_{n,k}))$.
Although the sequence $0,W_1,\ldots,W_n$ is not a random walk, it is
useful to
think of it as such for the purposes of estimating various probabilities.

Denote by $\HH_{n,k}$ the set of
vectors $(h_1,\ldots,h_n)$ of positive integers with
$0<h_1<\cdots<h_n=n+k$ and note that $|\HH_{n,k}|={{n+k-1}\choose{k}}$.
The sequence $(h_1(\bv_{n,k}),\ldots,h_n(\bv_{n,k}))$ is distributed as
a uniformly random element of~$\HH_{n,k}$.

%For $i\ge1$, $W_i$ is the sum of $h_i$ independent random variables
%with exponential distribution.
%We can thus think of the sequence $0,W_1,\ldots,W_n$ as sampling a
%random
%walk with exponential steps at times $0,h_1,\ldots,h_k$.
%Under $\WW{n,k}$,
%With $n$ and $k$ fixed, choose $(h_1,\ldots,h_k)$ from $\HH_{n,k}$
%uniformly at random, and independently choose the values of the
%$n+k$ exponential random variables. Under these conditions, let
%$\WW_{n,k}$ be the law of the sequence $\bW=(W_1,\ldots,W_n)$.

For $v \in T_n$, let $L_a=L_a(v)$ denote the event $\{W_i \ge(i/n)W_n
- a (i\le n)\}$.
A~vertex $v$ is called \textit{leading} if $L_{0}(v)$ holds,
and---informally---\textit{near-leading} if $L_{a}(v)$ holds for some
small $a$.
[We also will need to consider the event $R_{a}(v)=\{ W_i \le(i/n)W_n
+ a (i\le n)\}$,
and when this event occurs we say $v$ is ``near trailing.'']

If $M_n$ is not much larger than normal, $v$ is the vertex at level $n$
with minimal $S(v)$ and
$W_i\le(i/n)W_n-c$ for a large
$c$, then $M_i$ will be smaller than normal and this is rare. Hence,
with high probability $v$ will be a near-leading vertex.
On the other hand, near-leading vertices are rare---a given vertex in
$T_n$ is
near leading with probability $O(f(a)/n)$ for some function $f$. It
will turn out, as in
prior work~\cite{addario07brw},\vadjust{\goodbreak} that $\E M_n$ is within $O(1)$
of the smallest $x$ such that the expected number of leading nodes with
displacement at most $x$ is at least $1$.

In this section we develop estimates for the probability that vertices
of $T_{n,k}$ are
near leading. As in~\cite{addario07brw}, we also show that for a near-leading
vertex $v$, it is rare for $W_i(v)-(i/n)W_n(v)$ to be small if $i$
is far away from $0$ and far from $n$. This useful fact will play
an important role in the proof of Theorem~\ref{T1}.\looseness=-1

The next proposition, stated without proof, follows from the well-known
fact that
a Poisson sample becomes a uniform sample once conditioned on the
position of the $n$th point.

\begin{proposition}\label{expon}
For any positive real numbers $b_1,\ldots,b_n$ and $B$, and any $v \in T_n$,
\begin{eqnarray}
\P\bigl( W_i\ge b_i (i<n) | W_n=B \bigr) &=& \P\biggl( \frac{W_i}{W_n} \ge\frac
{b_i}{B} (i<n) \biggr), \\[-2pt]
\P\bigl( W_i\le b_i (i<n) | W_n=B \bigr) &= &\P\biggl( \frac{W_i}{W_n} \le\frac
{b_i}{B} (i<n) \biggr).
\end{eqnarray}
\end{proposition}

Proposition~\ref{expon} allows us to rescale the values $W_i$ to choose
a convenient value for $W_n$: for given $B'$, letting
$b_i'=b_i\cdot B'/B$, the proposition implies that
\[
\P\{ W_i\ge b_i (i<n) | W_n=B \} = \P\{ W_i\ge b_i' (i<n) |
W_n=B' \}.
\]
We will use this fact rather casually in what follows. We will also use
the following
variant of a well-known fact about cyclically exchangeable sequences.

\begin{proposition}\label{leading}
For any $S>0$,
\[
\P\{ L_{0}(\bv_{n,k}) | W_n=S \} =
\P\{ R_{0}(\bv_{n,k}) | W_n=S \} =\frac1{n}.
\]
\end{proposition}

\begin{pf}
For $0 \leq l < n$, let $W_{n+l}=W_n+W_l$. Then, for each
$0\le l< n$ and all $0 < j \leq n$, let $W_j^{(l)}=W_{j+l}-W_l$.
Then for all $l$, $W_n^{(l)}=W_n$. Furthermore,
each sequence $\bW^{(l)}=(W_1^{(l)},\ldots,W_n^{(l)})$ has
distribution $\WW_{n,k}$ and a.s. exactly one of them is leading by
the Cycle lemma~\cite{dvoretzky47problem}.
Similarly, exactly one of the sequences $\bW^{(l)}$ is ``trailing.''
\end{pf}

The following straightforward fact essentially says that conditioning
on any subset of the differences $h_1-h_0,\ldots,h_n-h_{n-1}$ breaks
the sequence into independent subsequences with distributions from the
same family. The proof is omitted.

\begin{fact}\label{condfact}
Fix integers $n \geq1$, $k \geq0$, and let $(W_1,\ldots,W_n)$ have
law $\WW_{n,k}$.
Then for any integers $1 \leq i \leq m \leq n$, and $1=n_0 < n_1 <
\cdots< n_m = n$,
conditional upon $h_i-h_{i-1}$, the sequence
\[
(W_{n_{i-1}+1}-W_{n_{i-1}},\ldots,W_{n_i}-W_{n_i-1})\vadjust{\goodbreak}
\]
has law $\WW_{n_i-n_{i-1},(h_i-h_{i-1})-(n_i-n_{i-1})}$, and is
mutually independent of $(h_1,\ldots,\break h_n)$,
of $(W_1,\ldots,W_{n_{i-1}})$, and of $(W_{n_i+1}-W_{n_{i}},\ldots
,W_n-W_{n-1})$.
\end{fact}

The next two lemmas are analogs of Lemmas 11 and 12
in~\cite{addario07brw}, and are proved using some of the same ideas.
Whereas lemmas in~\cite{addario07brw} use heavily the fact that a
random walk $0,S_1,\ldots,S_n$ can be broken into independent sub-walks
$0,S_1,\ldots,S_j$ and $0,S_{j+1}-S_j,\ldots,S_n-S_j$, in our situation
the analogous subsequences $0,W_1,\ldots,W_j$ and
$0,W_{j+1}-W_j,\ldots,W_{n}-W_j$ are not independent. We circumvent the
lack of independence by
instead using Fact~\ref{condfact}.

\begin{lemma}\label{nearleading}
Uniformly for $S>0$, $0\le k\le n$ and $a\ge0$,
\begin{eqnarray*}
\P\{ L_a(\bv_{n,k}) | W_n(\bv_{n,k})=S \} &\ll&
\frac{(an/S)^6+1}{n}, \\
\P\{ R_a(\bv_{n,k}) | W_n(\bv_{n,k})=S \} &\ll&
\frac{(an/S)^6+1}{n}.
\end{eqnarray*}
\end{lemma}

\begin{remark*} Most likely, the exponent ``6'' can be replaced with ``2,'' in
analogy with results from~\cite{addario07brw}
about ballot theorems for random walks.
\end{remark*}

Given that $L_a(\bv_{n,k})$ holds, it is likely that $W_i-(i/n)W_n$
remains large when~$i$ is far from 1 and far from $n$.
It is also likely that $h_j$ is not too large
when $j$ is small, and, similarly, $h_n-h_j$ is not large when $j$ is
near $n$.
The next two lemmas make this very precise.

For $v \in T_n$, define the events
\[
B_a(v) = \{ \exists m\in[a^{40},n-a^{40}] \dvtx
W_m(v) \le(m/n)W_n(v) + \min(m,n-m)^{1/40} \}
\]
and
\[
D_a(v) = \{ \exists j \dvtx h_j(v) > 3aj \mbox{ or } h_n(v)-h_j(v) >
3a(n-j) \}.
\]

\begin{lemma}\label{wellbehaved}
Uniformly for $0\le k\le n/2$, $n/10 \le S \le n$ and $a\ge1$,
\[
\P\{ L_a(\bv_{n,k}), B_a(\bv_{n,k}) | W_n(\bv_{n,k})=S \} \ll\frac
{1}{na^7}.
\]
\end{lemma}

\begin{lemma}\label{Dnka}
Uniformly for $0\le k\le n/2$, $n/10 \le S \le n$ and $a\ge0$,
\[
\P\{ L_a(\bv_{n,k}),D_a(\bv_{n,k}) | W_n(\bv_{n,k})=S \} \ll\frac
{e^{-a}}{n}.
\]
\end{lemma}

\newcommand{\tW}{\widetilde{W}}
\newcommand{\hW}{\widehat{W}}

\begin{pf*}{Proof of Lemma \protect\ref{nearleading}}
It suffices to prove the lemma when $a \geq10$.
We also assume $a \le n^{1/6}$, or else the conclusion is trivial.
Finally, in light of Proposition~\ref{expon}, we may assume without
loss of generality
that $S=n+k$, so that $n\le S\le2n$.\vadjust{\goodbreak}

Let
\begin{eqnarray*}
m &=& a^2,\qquad l=\lceil km/n\rceil,\qquad n'=n+2m,\qquad k'=k+2l,\\
\lam&=&\frac{n'+k'}{n'},\qquad a' = \frac{an\lam}{S}.
\end{eqnarray*}
We remark that $m,l \leq n^{1/3} $,
$a\lambda/2 \leq a' \leq a\lambda$,
and for $n$ large enough \mbox{$1 \leq\lambda\leq3$}.

%We consider a sequence $\bW'$ with law $\WW_{n',k'}$ and with
%associated $h_1',\ldots,h_{n'}'$.
By Proposition~\ref{leading},
\begin{equation}\label{nla}
A:= \P\{ L_0(\bv_{n',k'}) | W_{n'}(\bv_{n',k'})=\lam n' \} = \frac{1}{n'}.
\end{equation}

Now let $\bW'=(W'_1, \ldots, W'_{n'})$ be a sequence with law $\WW_{n',k'}$.
We bound $A$ from below by counting only sequences with
$h'_m=m+l$ and $h'_{n-m}=(n'+k')-(m+l)$. In this way, we can break
$\bW'$ into three subsequences, namely,
\begin{eqnarray*}
\mathbf{\tW},&&\qquad \mbox{where } \tW_j=W_j', \tilde{h}_j=h'_j
\ (1\le
j\le m), \\
\bW, &&\qquad\mbox{where } W_j=W_{j+m}'-W_m', h_j=h_{j+m}'-h_j'\ (1\le
j\le
n),\\
\mathbf{\hW}, &&\qquad\mbox{where } \hW_j=W_{n'}'-W_{n'-j}', \hat{h}_j=n'+k'
-h'_{n'-j}\ (1\le j\le m).
\end{eqnarray*}
That is, $\mathbf{\tW}$ captures the first $m$ steps, $\bW$ the next $n$
steps, and $\mathbf{\hW}$ the last $m$ steps taken in reverse order.

We'll work with four events:
\begin{eqnarray*}
E_1 &= &\{ h'_m=m+l, h'_{n'-m}=(n'+k')-(m+l)\}, \\
E_2 &= &\{ \tW_j \ge\lam j (j\le m), \tW_m-\lam m \in[a',2a']
\}
, \\
E_3 &=& \{ \hW_j \le\lam j (j\le m), \hW_m-\lam m \in
[-3a',-2a'] \}, \\
E_4(x) &=& \{ W_j \ge\lam j - x (j<n)\}.
\end{eqnarray*}
Given $E_1$, $\mathbf{\tW}$ and $\mathbf{\hW}$ have law $\WW_{m,l}$,
and $\bW$ has law $\WW_{n,k}$, and all three are independent.
Also given $E_1$, the events $E_2$, $E_3$ and $E_4(x)$ are independent. Thus,
\begin{eqnarray}
\label{nlb}\quad
A & \ge&\P\{ E_1 | W_{n'}'=\lambda n'\} \P\{ E_2 |
E_1,W_{n'}'=\lambda n' \} \P\{ E_3 | E_1,W_{n'}'=\lambda n' \}
\nonumber
\\[-8pt]
\\[-8pt]
\nonumber
&&{} \times\mathop{\inf_{a'\le x\le2a' }}_{ -3a' \le y \le-2a'}
\P\{E_4(x) | E_1,W_{n'}'=\lambda n', W_n = \lambda n - x - y \}.
\end{eqnarray}
Since $m+l = O(n^{1/3})$, if $k>0$, then a slightly tedious but routine
computation with Stirling's formula and (\ref{eq:tnk}) gives
\begin{eqnarray}\label{E1}
\P\{E_1 | W_{n'}' = \lambda n'\}& =& \frac{{ m+l-1\choose l}^2
{n+k-1\choose k}}{{n+2m+2l+k-1 \choose k+2l}}
\order\pmatrix{{m+l-1}\vspace*{2pt}\cr{l}}^2 \frac{k^{2l}
n^{2m}}{(n+k)^{2m+2l}}
\nonumber
\\[-8pt]
\\[-8pt]
\nonumber
&\order&
\frac{1}{l}.
\end{eqnarray}
When $k=0$, trivially $\P\{E_1 | W_{n'}' = \lambda n'\}=1$.
For the remainder of the proof we write $\P^{\mathrm{c}}\{ \cdot \}$
to mean $\P\{
\cdot
| E_1,W_{n'}' = \lambda n'\}$.
Next,
\begin{eqnarray} \label{E2givenE1}
\P^{\mathrm{c}}\{ E_2 \}& =& \P^{\mathrm{c}}\{ \tW_j \ge\lam j
(j<m) | \tW_m - \lam m \in [a',2a'] \}
\nonumber
\\[-8pt]
\\[-8pt]
\nonumber
&&{}\times \P^{\mathrm{c}}\{ \tW
_m-\lam m \in[a',2a'] \}.
\end{eqnarray}
Given that $W'_{n'}=\lam n'$ and $h_m' = m+l$, $\tW_m$ has distribution
\[
\lambda n' \cdot\operatorname{Beta}(m+l,n+k-m-l)
\]
and, in particular, has mean
\[
\lambda n' (m+l)/(n+k) = \lambda m + O\biggl( \frac{m^2}{n} \biggr) =
\lambda m + O(1)
\]
and variance
\[
(\lambda n')^2 \frac{(m+l)(n+k-m-l)}{(n+k)^2 (n+k+1)} = O(m).
\]
Since $a' \ge\frac{a}{2} \ge\frac12\sqrt{m} $, it follows from the
definition of a Beta random variable that
the second probability on the right-hand side of (\ref{E2givenE1}) is $\gg
1$. %by the Central Limit Theorem and $a/2\le a'\le3a$.
Applying Proposition~\ref{expon} followed by Proposition~\ref
{leading}, the first factor on the right-hand
side of (\ref{E2givenE1}) is
\begin{eqnarray*}
&\ge&\inf_{a'\le x\le2a'} \P^c \{ \tW_j \ge\lam j (j<m) | \tW_m
=\lam m+x \}\\
&\ge&\inf_{a'\le x\le2a'} \P\{ L_0(\bv_{m,l}) | W_m(\bv_{m,l})
=\lam m+x \} = \frac{1}{m}.
\end{eqnarray*}
Therefore,
\begin{equation}\label{E2}
\P^c \{E_2\} \gg\frac{1}{m} = \frac{1}{a^2}.
\end{equation}
Similarly,
\begin{equation}\label{E3}\qquad
\P^c \{ E_3 \} \gg\inf_{-3a' \le y\le-2a'} \P\{ L_0(\bv_{m,l}) |
W_m(\bv_{m,l}) =\lam m+y \} = \frac{1}{m} = \frac{1}{a^2}.
\end{equation}
Last, for $a'\le x\le2a'$ and $-3a' \le y \le-2a'$, Proposition
\ref{expon} yields
\begin{eqnarray}\label{E4}
\P^c \{ E_4(x) | W_n = \lambda n - x - y \} &= &\P\biggl\{ \frac
{W_j}{W_n} \ge\frac{\lam j-x}
{\lam n-x-y}\ (j\le n) \biggr\} \nonumber\\
&\ge&\P\biggl\{ \frac{W_j}{W_n}
\ge
\frac{j}{n}-\frac{a}{S}\ (j\le n) \biggr\} \\
&=& \P\{ L_a(\bv _{n,k}) |
W_n(\bv_{n,k})=S \}.\nonumber
\end{eqnarray}
Together, \eqref{nla}--\eqref{E4} imply
\[
\frac{1}{n}\gg\frac{1}{a^6} \P\{ L_a(\bv_{n,k}) | W_n(\bv
_{n,k})=S \},
\]
which proves the first assertion of the lemma. The proof of the second
part is identical.
\end{pf*}

\begin{pf*}{Proof of Lemma \protect\ref{wellbehaved}}
Fix $k$, $S$ and $a$ as in the statement of the lemma.
We write $W_m=W_m(\bv_{n,k})$, $h_m=h_m(\bv_{n,k})$ and so on.
If $a^{40}>n/2$, then there
is nothing to prove so we assume $a^{40}\le n/2$.
For $a^{40} \le m \le n-a^{40}$ and $l\ge0$, let
\[
A_{m,l} = \P\biggl\{ L_a(\bv_{n,k}), W_m \le\frac{m}{n}S+ \min
(m,n-m)^{1/40} | W_n=S, h_m=m+l \biggr\}.
\]
Break $(W_1,\ldots,W_n)$ into two sequences: $\tW_j=W_j$ for $j\le m$,
and $\hW_j=W_n-W_{n-j}$ for $j\le n-m$ (the latter being the final
$n-m$ steps taken in reverse).
Given $h_m=h_m-h_0$, these sequences are independent by Fact~\ref{condfact}.
We write $\P^{\mathrm{c}}\{ \cdot \}$ for the conditional
probability measure $\P\{
\cdot| h_m=m+l \}$, and
$\E^{\mathrm{c}}\{ \cdot \}$ for the corresponding expectation operator.
Also, let $\lam=\frac{n+k}{n}$.

Suppose first that $a^{40} \le m\le n/2$. Put $b=m^{1/40} \frac{n+k}{S}$
and $a'=a \frac{n+k}{S}$. Note that
$\E^{\mathrm{c}}\{ \tW_m | W_n=S \} = S \cdot(m+l)/(n+k)$.
Rescaling by $(n+k)/S$ (this is allowed by the comment just after Proposition
\ref{expon}), by the definitions of $b$ and $a'$ we have
\begin{eqnarray}\label{amlbound}
\qquad A_{m,l} &\le&\P^{\mathrm{c}}\{ \tW_m-\lam m \in[-a',b] |
W_n=\lambda n \}\nonumber
\\
& &{}\times\sup_{-a'\le x\le b} \P^{\mathrm{c}}\{ \tW_j\ge\lam j -
a' (j<m) | \tW_m=\lam m +x \} \\
&&{} \times\sup_{-b\le x\le a'} \P^{\mathrm{c}}\{ \hW_j \le\lam j +
a' (j<n-m) | \hW_{n-m}=\lam(n-m)+x \}.\nonumber
\end{eqnarray}
Given that $W_n=\lambda n$ and $h_m=m+l$, $\tW_m$ has distribution
$\lambda n \cdot\operatorname{Beta}(m+l,k+n-m-l)$
and so the first factor on the RHS of (\ref{amlbound}) is $O(b/\sqrt
{m})$ uniformly in $l$ and in $n$.
Applying Proposition~\ref{expon} and the first inequality of Lemma
\ref{nearleading}, the second factor on the RHS of (\ref{amlbound}) is
\begin{eqnarray*}
&\le&\P^{\mathrm{c}}\biggl\{ \frac{\tW_j}{\tW_m} \ge\frac{j}{m} -
\frac {a'+bj/m}{m\lam +b}\ (j<m) \biggr\} \\
&\le&\P^{\mathrm{c}}\biggl\{ \frac
{\tW_j}{\tW_m} \ge\frac{j}{m} - \frac {a'+b}{m\lam+b}\ (j<m) \biggr\}\\
&=&\P\{ L_{a'+b}(\bv_{m,l}) | W_m(\bv_{m,l})=m\lam+b \}\ll
\frac
{(a'+b)^6+1}{m} \ll\frac{b^6}{m},
\end{eqnarray*}
so the product of the first two factors on the right-hand side of (\ref
{amlbound}) is $O(b^7 m^{-3/2})$.
Similarly, by Proposition~\ref{expon} and the second inequality of\vadjust{\goodbreak}
Lem\-ma~\ref{nearleading}, the third factor
on the RHS of (\ref{amlbound}) is
\[
\le\P\biggl\{ \frac{\hW_j}{\hW_{n-m}} \le\frac{j}{n-m} + \frac
{a'+b}{\lam(n-m)-b}\ (j<n-m) \biggr\} \ll\frac{b^6}{n-m} \ll\frac{b^6}{n}.
\]
Combining these bounds, we obtain that when $a^{40} \leq m \leq n/2$,
$A_{m,l} \ll b^{13}/(nm^{3/2})$.
The estimation of $A_{m,l}$ with $m>n/2$ is identical, by reversing the roles
of $\mathbf{\tW}$ and $\mathbf{\hW}$. Therefore,
\begin{eqnarray*}
\P\bigl( L_a(\bv_{n,k}), B_a(\bv_{n,k}) | W_n=S \bigr) &\ll&
\sum_{a^{40}\le m\le n/2} \sum_{l\ge0} \P\{ h_m=m+l \} A_{m,l} \\
&\ll&\frac{1}{n} \sum_{a^{40}\le m\le n/2} \frac{m^{13/40}}{m^{3/2}}
\ll\frac{1}{na^7}.
\end{eqnarray*}
\upqed\end{pf*}

\begin{pf*}{Proof of Lemma \protect\ref{Dnka}}
As before, we write $W_n=W_n(\bv_{n,k})$, $h_j=h_j(\bv_{n,k})$ and so on.
We may assume $a\ge10$, or else the conclusion follows from Lemma~\ref{nearleading}.
We also assume $k\ge1$, or else $h_j=j$ for every $j$
and $D_a(\bv_{n,k})$ is impossible.
For fixed $j$, given $h_j$, the sequence
$(W_1,\ldots,W_{n})$ breaks into two independent sequences $\mathbf
{\tW}$,
consisting of the first
$j$ steps, and $\mathbf{\hW}$, consisting of the last $n-j$ steps taken
in reverse.
If $W_n=S$ and $L_a(\bv_{n,k})$ holds, then
there is an integer $b\ge-a-1$ so that $\tW_j-\frac{j}{n}S \in[b,b+1]$.
Consequently, $\hW_{n-j} - \frac{n-j}{n}S \in[-b-1,-b]$.

Fix $h$ such that $h > 3aj$ and suppose that $h_j=j$---note
that in this case $j< \frac{n+k}{3a} \le\frac{n}{20}$.
Given that $h_n=h$ and $W_n=S$, $\tW_j$ has distribution $S \cdot
\operatorname{Beta}(h,n+k-h)$.
Since $k \leq n/2$ and $S \leq n$, it is then straightforward to check that
$\P\{ \tW_j\ge b | h_n=h, W_n=S \}\le e^{-b/4}$ for $b\ge4h$.
We also have
\begin{eqnarray*}
&&\mathbb{P}\biggl\{\hW_i\ge\frac{i}{n}S-(a+b) (i\le n-j)
\Big|\hW_{n-j}-\frac{n-j}{n}S\in[-b-1,-b],\\
&&\hspace*{228pt}{}h_j=h,W_n=S\biggr\} \\
&&\qquad \le\P\biggl\{ L_{2a+b}(\bv_{n-j,k-h+j}) \Big| W_{n-j}(\bv
_{n-j,k-h+j})-\frac{n-j}{n}S\in[-b-1,-b] \biggr\} \\
&&\qquad \ll\frac{(2a+b)^6}{n}
\end{eqnarray*}
by Lemma~\ref{nearleading} if $b\le n^{1/6}$, and trivially otherwise.
Summing on $b$, we find that
\[
\P\{ L_a(\bv_{n,k}) | h_j=h, W_n=S \} \ll\sum_{-a-1\le b\le4h}
\frac{(2a+b)^6}{n} + \sum_{b>4h} \frac{(2a+b)^6}{ne^{b/4}}
\ll\frac{h^7}{n}.
\]
Note that $(h_1,\ldots,h_n)$ is independent of $W_n$ and
so $ \P\{ h_j=h | W_n=S \} = \P\{ h_j=h \}$.
Since $h-j\le k\le n/2$, by Stirling's formula,
\begin{eqnarray} \label{hjstirlingbound}
\P\{ h_j=h \} &=& \pmatrix{{h-1}\vspace*{2pt}\cr{h-j}} \frac{{{n+k-h-1}\choose{k-h+j}}}{
{{n+k-1}\choose{k}}} \nonumber\\
&\le&\frac{h^j}{j!} \cdot\frac{(n-1)\cdots(n-j) \cdot k \cdots
(k-h+j+1)} {(n+k-1)\cdots(n+k-h)} \\
&\le&\biggl( \frac{eh}{j} \biggr)^j \biggl(\frac{k}{n}\biggr)^{h-j} \le(6ae)^{h/(3a)} 2^{-h}
< e^{-h/2},\nonumber
\end{eqnarray}
the last inequality holding at least for $a \geq5$ (which we have assumed).
Summing over $h>3aj$, then over $j$, we find that
\begin{equation}\label{D1}
\qquad \P\{ L_a(\bv_{n,k}), \exists j\dvtx h_j>3aj | W_n=S \} \ll\frac{1}{n}
\sum_{j\ge1} \sum_{h>3aj} h^7 e^{-h/2} \ll\frac{e^{-a}}{n}.
\end{equation}

Next, suppose $h=h_n-h_{j}>3a(n-j)$, in which case $n-j < \frac{n}{20}$.
Let $b'=W_j-\frac{j}{n}S$.
Since $W_{i+1}\ge W_i$ for all $i$, $W_j \leq S$ and so $b' \le\frac
{n-j}{n}S \le n-j$.
Also, in order for $L_{n,k}(a)$ to occur, we must have $b' \geq-a$.
Thus, writing
$\mathcal{I} = [-a,n-j]$, and ignoring the last $n-j$ steps of $W$ for
an upper bound, we have
\begin{eqnarray*}
&& \P\{ L_a(\bv_{n,k}) | h_n-h_j=h, W_n=S \} \\
&&\qquad \le\sup_{b' \in\mathcal{I}} \P\biggl\{ \tW_i\ge\frac{i}{n}S-a (i\le
j) \Big| \tW_j=\frac{j}{n}S+b',h_n-h_j=h \biggr\}\\
&&\qquad \le\sup_{b' \in\mathcal{I}} \P\biggl\{ L_{a+b'}(\bv_{j,n+k-h}) \Big|
W_j(\bv_{j,n+k-h})=\frac{j}{n}S + b' \biggr\}.
\end{eqnarray*}
Note that $a \leq n/2$ [or else $3a > 3n/2 > n+k$ and $D_a(\bv_{n,k})$
is impossible].
Since $j \geq\frac{19}{20} n$ and $b' \geq-a \geq-n/2$,
by Lemma~\ref{nearleading} and straightforward manipulations, the last
probability is $O(\frac{1}{n}(a+b')^6)=O(\frac{1}{n}(a+n-j)^6)$.
%The second factor is $\le1$ and
Also, $\P\{ h_n-h_{j}=h \}=\P\{ h_{n-j}=h \}<e^{-h/2}$ by the same
calculation as in (\ref{hjstirlingbound}). Summing
over $h>3a(n-j)$ and $j\le n-1$ gives
\[
\P\{ L_a(\bv_{n,k}), \exists j\dvtx h_n-h_{j}>3a(n-j) | W_n=S \} \ll
\frac
{e^{-a}}{n}.
\]
Together with \eqref{D1}, this completes the proof.
\end{pf*}

%%%%%%%%%%%%%%%%%%%%%%%%%%%%%%%%%%%%%%%%%%%%%%%%%%%
%
%s5 ###
\section{\texorpdfstring{The lower bound in Theorem \protect\ref{T1}}{The lower bound in Theorem 1.1}}\label{sec5}
%
%%%%%%%%%%%%%%%%%%%%%%%%%%%%%%%%%%%%%%%%%%%%%%%%%%%
We continue to adopt the notational conventions from the previous section.
Let $c$ be a sufficiently large positive constant, and $b=e^{c/3}$. Let
\[
Y_n = \bigcup_{|k-n/e|\le\sqrt{6n\log n}} T_{n,k}
\]
and put $m_n=\frac{n}{e}+\frac{3\log n}{2e}$.
If $M_n\le m_n - c$, then one of the following must occur:
\begin{longlist}[(iii)]
\item[(i)] For some $v\in T_n$, $S(v)\le m_n - \log n$;
\item[(ii)] For some $k$ satisfying $|k-n/e|>\sqrt{6n\log n}$ and some
$v\in T_{n,k}$,\break \mbox{$S(v) \le m_n$};
\item[(iii)] For some $v\in Y_n$, $m_n-\log n \le S(v)\le m_n$ and $W_i \le
(i/n)W_n-\log n$ for some $i$;
\item[(iv)] For some $v\in Y_n$, $m_n-\log n\le S(v)\le m_n-c$ and
$W_i \ge(i/n)W_n-b$ for all $i$;
\item[(v)] For some $v\in Y_n$ and some integer $a\in[b,\log n+1]$,
$m_n-\log n\le S(v)\le m_n$,
$W_i\ge(i/n)W_n-a$ for all $i$ and $W_j< (j/n)W_n-(a-1)$ for some $j$
(write $F_{a,j}$ for the event that this
occurs for a given $a$ and $j$ with $j$ minimal, and note that these
events are disjoint).
\end{longlist}

By Lemma~\ref{ETnx} and Stirling's forumula, the probability of (i) is
at most
$\E\{T_n(m_n-\log n)\}= O(n^{1-e})$.
The probability of (ii) is $O(n^{-1/2})$ by Lem\-ma~\ref{lemma_ab}(b). If (iii) occurs, then $M_i \le(i/n)m_n-\log n$,
and this
happens with probability at most $\E\{T_i((i/n)m_n-\log n)\}$, which is
$O(n^{3/2-e}i^{-1/2})$
by Lemma~\ref{ETnx}.
Summing on $i$, we find that (iii) occurs with probability $O(n^{ 2-e})$.

To bound the probability of the event in (iv), we write
$E_k$ for the event that there is $v \in T_{n,k}$ for which $m_n - \log
n \leq S(v) \leq m_n-c$ and
$W_i \geq(i/n) W_n - b$ for all~$i$, so that by a union bound and
Lemma~\ref{nearleading}, the probability
of (iv) is at most
\begin{eqnarray}\label{randomtomany}
&& \sum_{|k-n/e| \leq\sqrt{6n\log n}} \P\{ E_k \} \nonumber\\
&&\qquad\le \sum_{|k-n/e| \leq\sqrt{6n\log n}} |T_{n,k}| \P\{ m_n-\log n
\leq S(\bv_{n,k}) \leq m_n-c, L_b(\bv_{n,k}) \} \nonumber\\
&&\qquad\ll \sum_{|k-n/e| \leq\sqrt{6n\log n}} |T_{n,k}| \P\{ m_n-\log n
\leq S(\bv_{n,k})\leq m_n-c \} \cdot\frac{b^6}{n} \\
&&\qquad\leq \frac{b^6}{n} \cdot\E\{T_n(m_n-c)\} \nonumber\\
&&\qquad\ll e^{(2-e)c}.\nonumber
\end{eqnarray}
[This line of argument will arise again in bounding (v), and we will
omit the details.]

Finally, we bound (v). To do so, we are forced to separately treat $j$
in three different ranges.
First suppose $j \leq a^{40}$. If $F_{a,j}$ occurs, then $M_j \le
(j/n)m_n-a$, the probability of which is $O(j^{-1/2}e^{-ea})$
by Lemma~\ref{ETnx}. Summing on $a$ and on $j \leq a^{40}$ gives a
total probability
of $O(e^{-2b})$ for this range of parameters.

Next suppose that $a^{40}<j<n-a^{40}$, so that $(\min(j,n-j))^{1/40}
\geq a$. If
$F_{a,j}$ occurs, then for some $v \in Y_n$, $L_a(v)$ and $B_a(v)$ both occur.
Note that for $n$ large enough
$n/10 \leq m_n-\log n \leq m_n \leq n$, and for all $k$ for which
$T_{n,k} \subseteq Y_n$ we have $0 \leq k \leq n/2$.
Thus, for such $n,k$ and $a$, we may apply Lemma~\ref{wellbehaved} to
see that
\[
\P\{ L_a(\bv_{n,k}),B_a(\bv_{n,k}) | m_n-\log n \leq W_n(\bv_{n,k})
\leq m_n \} \ll\frac{1}{na^7}.
\]
Further, the expected number of $v \in Y_n$ with $S(v)\le m_n$ is
$O(n)$ by Lemma~\ref{ETnx}.
By these two bounds and a reprise of the argument leading to (\ref
{randomtomany}), we see that for a given $a$, the probability of
$\bigcup_{j \in[a^{40},n-a^{40}]} F_{a,j}$ is $O(1/a^{7})$ and summing
over integers $a \in[b,\log n+1]$ gives a total probability of
$O(1/b^6)=O(e^{-2c})$.

Now suppose $F_{a,j}$ holds with $j \in[n-a^{40},n]$ and $a \in[b,
\log n+1]$.
By the definition of $F_{a,j}$, letting $w$ be the unique ancestor of
$v$ in $T_j$, the event $L_{a-1}(w)$ also occurs.
Since $j \geq n-(\log n+1)^{40}$, for $n$ sufficiently large
$|m_j - (j/n) m_n| \leq1$ and, hence, $S(w) \leq m_j+1-(a-1)$.
On the other hand, for any integer $k' \geq1$,
by Lemma~\ref{nearleading} we have
\[
\P\{ W_j(\mathbf{v}_{j,k'}) \leq m_j+2-a,L_{a-1}(\bv_{j,k'}) \}
\ll\frac{a^6}{j} \P\{ W_j(\mathbf{v}_{j,k'}) \leq m_j+2-a \}.
\]
By Lemma~\ref{ETnx}, it follows that
\[
\P\{ F_{a,j} \} \ll\frac{a^6}{j} \E|T_j(m_j+2-a)|
\ll a^6 e^{-ea}.
\]
Summing first over $j \in[n-a^{40},n]$, then over $a \in[b,\log
n+1]$, we see that
the probability $F_{a,j}$ occurs for any $a$ and $j$ in this range is
\[
\ll b^{46}e^{-eb} = \exp\{(46/3) c - e^{1+c/3}\} < e^{-2c}
\]
as long as $c$ is large enough.
Combining the three ranges, we obtain that (v) occurs with probability
$\ll e^{-2c}$.
Altogether, the probability that one of (i)--(v) holds is $\ll e^{(2-e)c},$
which is less than $1/2$ if $c$ is chosen large enough. Hence,
$\tM_n\ge m_n-c$.

%%%%%%%%%%%%%%%%%%%%%%%%%%%%%%%%%%%%%%%%%%%%%%%%%%%%%

%s6 ###
\section{\texorpdfstring{The upper bound in Theorem \protect\ref{T1}}{The upper bound in Theorem 1.1}}\label{sec6}

%%%%%%%%%%%%%%%%%%%%%%%%%%%%%%%%%%%%%%%%%%%%%%%%%%%%%

%% FIGURE - problems TeX-ing it

For the upper bound for $\operatorname{median}(M_n)$,
we use a second-moment method.
By the Cauchy--Schwarz inequality, for any nonnegative
random variable $X$,
\begin{equation}\label{eq:smm_j}
\P\{X>0\} \geq\frac{[\E{X}]^2}{\E{X^2}}.
\end{equation}
When $X$ is the size of some random subset $\mathcal{X}$ of a ground
set $V_0$, we may rewrite~(\ref{eq:smm_j}) using the
fact that
\[
\E{X^2} = \sum_{v,w \in V_0} \P\{v \in\XX,w \in\XX\} = \sum_{v
\in
V_0} \mathbb E[X|v \in\XX]\P\{v \in\XX\},
\]
so that
\begin{equation}\label{eq:smm_set}\qquad
\P\{X>0\} \geq\frac{[\E{X}]^2}{\sum_{v \in V_0} \mathbb E[X|v \in
\XX]\P\{
v \in\XX\}} \ge\frac{\E{X}}{\sup_{v \in V_0} \mathbb E[X|v \in
\XX]}.
\end{equation}

Let $a$ be a large positive constant.
Let $V_0=Y_n$, where $Y_n$ is defined as in the previous section,
and let $\XX$ be the set of nodes in $v \in Y_n$ satisfying
\begin{longlist}[(iii)]
\item[(i)] $m_n-1 \le S(v) \le m_n$,
\item[(ii)] $L_a(v)$,
\item[(iii)] neither
$B_a(v)$ nor $D_a(v)$.
\end{longlist}
Taking $X=|\XX|$, by Lemma~\ref{lemma_ab}(b), plus Lemmas \ref
{leading},~\ref{wellbehaved} and
\ref{Dnka}, we have
\begin{eqnarray}\label{EX}
 \E X &\ge&\E[ T_n(m_n)-T_n(m_n-1) ] \biggl( \frac{1}{n} - O\biggl(\frac{1}{a^7n}\biggr)
- O\biggl(\frac{e^{-a}}{n}\biggr) \!\biggr) - O\biggl(\frac{1}{n^{1/2}}\biggr)
\nonumber\hspace*{-35pt}
\\[-4pt]
\\[-12pt]
\nonumber
& \gg&1\hspace*{-35pt}
\end{eqnarray}
if $a$ is chosen large enough.

Recall that for all $v \in Y_n$, $|k(v)-n/e| \leq\sqrt{6n \log n}$.
For fixed $v \in Y_n$,
we need to estimate $\E\{ X | v \in\XX \}$.

The definitions of the coming two paragraphs are for the most part
depicted in Figure~\ref{vardefs}.
Write $j=j(v,v')$ for the integer $0 \leq j < n$ such that
$v$ and $v'$ are descendants of two distinct children of some node
$w=w(v,v') \in T_j$
[and let $j(v,v')=n$ if $v=v'$]. In other words, $j(v,v')$ is the
generation of the most recent common ancestor of $v$ and $v'$.
Supposing $0\le j(v,v') \le n-1$, let $x$ be the unique child of~$w$ on
the path from $w$ to $v'$.

Also, write $\bW=\bW(v)$ and $\bW'=(W_1',\ldots,W_{n}')=\bW(v')$.
Let $g=n-(j(v,v')+1)$, let $\tW_i =\tW_i(v,v')= W'_{j+i+1}-W'_{j+1}$
for $1 \leq i \leq g$,
and let $\mathbf{\tW}=(\tW_1,\ldots,\tW_g)$, so, in particular,
$W_{n}'=W'_{j+1}+\tW_g$.

Finally, let $k'=k(v')$, let $k_1=k_1(v,v')=k(x)-k(w)$ and let
$k_2=k'-k(x)$, so $k_1+k_2=k'-k(w)$.
Note that once $g$ and $k_2$ are fixed,
$\mathbf{\tW}$ is independent of $\bW$ and has law
$\WW_{g,k_2}$.

For integers $j$, $0 \leq j \leq n$, let $\FF_j = \FF_j(v) = \{v' \in
\XX,j(v,v')=j\}$ and let $F_j=F_j(v) = \E\{|\FF_j| | v \in\XX\}$.
Clearly, $F_n=1$, as $j=n$ implies $v=v'$.
% Let $\bW$ be the sequence of displacements associated with $v$, and
%let
% $\bW'$ be the sequence associated with $v'$.
%Also, let $\mathbf{\tW}$
%be the sequence of the final $g:=n-j-1$ displacements in $v'$, i.e.
%$\tW_i=W'_{j+1+i}-W'_{j+1}$.
%Let $h^*=h(v_0)$ and $k_0=k(v_0)$.

%f1 ###
\begin{figure}

\includegraphics{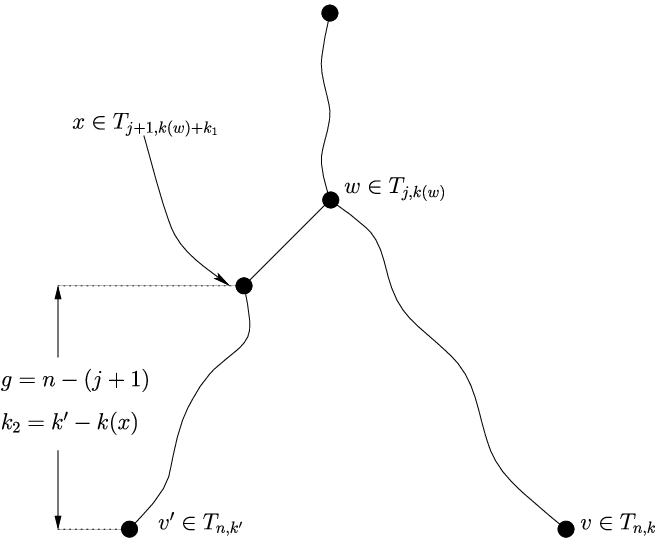}

\caption{An illustration of some key definitions from the proof of the upper bound of Theorem \protect\ref{T1}.}\label{vardefs}
\end{figure}

Now fix $v'$. If $v'\in\mathcal{X}$, then
by (i), (ii) and (iii), we have
\[
k(x) \leq3a(j+1),\qquad k'-k(w) \leq3a(g+1)=3a(n-j)
\]
and so
\begin{equation}\label{k1}
k_1+k_2 \leq3a(n-j),\qquad k_1 \leq\min\bigl(3a(j+1),3a(g+1)\bigr)
\end{equation}
%
%k_1+k_2 \le\min( 3a(j+1), 3a (g+1) )
and if $v \in\XX$, then, with $j=j(v,v')$, we have
\begin{equation}\label{Wj}
 W_j \ge\frac{j}{n}(m_n-1) +
\cases{
(-a), &\quad $\mbox{whatever the value of } j, $\vspace*{2pt}\cr
\min(j,n-j)^{1/40}, & \quad $\mbox{if }a^{40}\le j\le n-a^{40}.$}\hspace*{-35pt}
\end{equation}

Consider separately four ranges of $j$. First, if $n-a^{40}\le j\le n-1$,
then for sufficiently large $n$,\vadjust{\goodbreak} \eqref{k1} implies that $k_1+k_2 \le3a(n-j)$,
so $F_j$ is deterministically at most
\[
\sum_{l \leq3a(n-j)} |T_{n-j,l}| = \sum_{l\le3a(n-j)} \pmatrix{
{n-j+l-1}\vspace*{2pt}\cr{l}} \le3a^{41} (a^{40}+3a^{41})
^{a^{40}}.
\]
Hence, recalling that $a$ is now a fixed, large constant,
\begin{equation}\label{biggestj}
\sum_{n-a^{40}\le j\le n} F_j \ll1.
\end{equation}

Next, let $r=(2\log n)^{40}$. If $n-r<j\le n-a^{40}$, then for $n$
sufficiently large,
$j \ge n-j=g+1$, and \eqref{Wj} implies that
in order to have $W_n' \leq m_n$ we must have\looseness=-1
\[
\tW_{g} \le\frac{g+1}{n} m_n - g^{1/40} + 1\le g/e-g^{1/40}+2,
\]\looseness=0
the second inequality holding for sufficiently large $n$.
For fixed $k_1$, by Lem\-ma~\ref{ETnx} we thus have
\begin{eqnarray*}
\E\bigl\{ |\{v' \in\XX, j(v,v')=j,k_1(v,v')=k_1\}| | v \in\XX \bigr\} & \leq&
\E
T_g(g/e-g^{1/40}+2) \\
& \ll&\exp[ -e g^{1/40} ].
\end{eqnarray*}
Using \eqref{k1} to bound $k_1$ and summing over $j$ yields
\begin{equation}\label{bigj}
\sum_{n-r<j\le n-a^{40}} F_j \ll\sum_{a^{40}\le g \le r}
a(g+1) \exp[ -e g^{1/40} ] \ll1.
\end{equation}

Next, suppose $r\le j\le n-r$. By \eqref{Wj}, in order to have $W_n'
\leq m_n$, it must be that
\[
\tW_{g} \le\frac{g+1}{n} m_n - \min(j,n-j)^{1/40} + 1 \le
\frac{g}{e} - \log n.
\]
Since we also require $k_1(v,v') \leq3an$ by \eqref{k1}, we have
$F_j \le3an \E T_g(g/e-\log n)\ll1/n^2$ for this range of $j$, and, hence,
\begin{equation}\label{midj}
\sum_{r\le j\le n-r} F_j \ll\frac{1}{n}.
\end{equation}

Finally, suppose $0\le j\le r$. Here $g\ge n-r-1=n+O((\log n)^{40})$,
and since $L_a(v)$ holds by assumption, if $v \in\XX$, then
\[
W_j \geq\frac{j}{n} W_n - a > \frac{j}{n} m_n - (a+1).
\]
For each integer $b \in[-(a+1),2\log n)$, let $E_b$ be the event that
$W_j - (j/n) m_n \in[b,b+1)$. Also, let $E^*$ be the event that $W_j -
(j/n)m_n \geq\lceil2 \log n \rceil$. The events $\{E_b\dvtx
-(a+1)\leq b <
2\log n\}$ and $E^*$ together partition the event $\{v' \in\FF_j(v)\}
$, so by conditioning
\begin{equation}\label{conditioningbound}\quad
F_j \leq\max\Bigl( \E\{ |\FF_j| | v \in\XX,E^* \}, \max_{-(a+1)
\leq b
< 2\log n} \E\{ |\FF_j| | v \in\XX,E_b \}\Bigr).
\end{equation}

If $W_j\ge(j/n)m_n+2\log n$, then to have $v' \in\FF_j$, we must have
$\tW_g(v')\le g/e-\log n$
so, as in the case $r \leq j \leq n-r$, we have
\[
\E\{ |\FF_j| | v \in\XX,E^* \} \ll\frac{1}{n^2}.
\]
Now suppose $W_j-(j/n)m_n \in[b,b+1]$, where $b$ is an integer satisfying
$-a-1 \le b\le2\log n$. Note that if $b< (j^{1/40}-2)$ and $a^{40}\le
j\le r$, then $\FF_j(v)$ is necessarily empty due to $B_a(v)$,
so for such $j$ and $b$, $\E\{ |\FF_j| | v \in\XX,E_b \}=0$. For
the rest,
we further subdivide $\FF_j$, writing $\FF_{j,l}=\{v' \in\XX,
j(v,v')=j, k_1(v,v')=l\}$.
By \eqref{k1} we have
\[
\E\{ |\FF_j| | v \in\XX,E_b \}=\sum_{l \leq3a(j+1)} \E\{ |\FF
_{j,l}| | v \in\XX,E_b \}.
\]
Suppose additionally that
$W'_{j+1}(v')-W'_j(v')\in[\Delta,\Delta+1]$, where $\Delta$ is a
nonnegative integer.
Since $W_j(v')=W_j(v)$, in order to have $v' \in\FF_j$, by (i)
we require\footnote{The $m_n/n$ terms come from the ``skipped step''
from $W_j'$ to $W_{j+1}'$,
and the $(b+\Delta+3)$ comes from $b+1$, $\Delta+1$, and the
requirement that $S(v) \geq m_n-1$.}
\[
\tW_{g}-\frac{g}{n} m_n \in[m_n/n-(b+\Delta+3),m_n/n-(b+\Delta)].
\]
Since $0 \leq m_n/n<1$ and, for $n$ sufficiently large, $m_g-1 \leq
(g/n)m_n \leq m_g$,
this implies that, writing $\iii=[-(b+\Delta+4),-(b+\Delta-1)]$, we
must have
\[
\tW_{g} - m_g \in\iii.
\]
By (i) and (ii), we also require
\[
\tW_{i} \ge\frac{i}{n}W_n-b-\Delta-a-2 \ge\frac{i}{g}m_g -
b-\Delta
-a-3\qquad (i\le g).
\]
This implies that for all $i \le g$,
\[
\tW_i \geq\frac{i}{g} \tW_g - \max(b+\Delta+a-3,a-2).
\]
None of this depends on $l$,
so for any $0 \leq l \leq3a(j+1)$, writing $m=\max(b+\Delta+a-3,a-2)$,
\begin{eqnarray*}
&&\E\{ |\FF_{j,l}| | v \in\XX,E_b \}\\
&&\qquad\leq
\mathop{\sum_{1 \leq k_2\leq3a(j+1)}}_{\Delta\geq0} \E|\{v \in
T_{g,k_2} \dvtx S(v) -m_g \in\iii,L_{m}(v)\}| \\
&&\qquad\leq\mathop{\sum_{1 \leq k_2\leq3a(j+1)}}_{\Delta\geq0}
\Bigl(\E\bigl\{ T_{g,k_2}\bigl(m_g-(b+\Delta-1)\bigr) \bigr\} \\
&&\hspace*{65pt}\qquad{} \times\sup_{x \in\iii} \P\{ L_m(\bv_{g,k_2}) |
W_g(\bv _{g,k_2}) =m_g+x \} \Bigr) \\
&&\qquad\ll\sum_{\Delta\geq0} \E\bigl\{ T_{g}\bigl(m_g-(b+\Delta-1)\bigr) \bigr\} \cdot j
\cdot
\frac{\max(b+\Delta+a-3,a-2)^6}{n}\\
&&\qquad\ll\sum_{\Delta\geq0} ne^{-e(b+\Delta)} \cdot j \cdot\frac{\max
(b+\Delta+a-3,a-2)^6}{n} \\
&&\qquad\ll je^{-eb} (a+|b|)^6,
\end{eqnarray*}
the third-to-last line by Lemma~\ref{nearleading} and the
second-to-last by Lemma~\ref{ETnx}.
Summing over $0 \leq l \leq3a(j+1)$, it follows that
\[
\E\{ |\FF_j| | v \in\XX,E_b \} \ll j^2 e^{-eb} (a+|b|)^6.
\]
For $j \leq a^{40}$ this is $O(1)$ uniformly in $b$.
When $j > a^{40}$ we also have $b \geq j^{1/40}-2$
and for such $j$, the above bound is $O(j^3e^{-ej^{1/40}})$.
By (\ref{conditioningbound}) it follows that
for such $0 \leq j \leq r$,
\[
F_j \ll
\cases{1, & \quad $\mbox{if } j \leq a^{40},$\vspace*{2pt}\cr
\max(n^{-2}, j^3\exp(-ej^{1/40})), & \quad $\mbox{if } j > a^{40}.$}
\]
%
% the expected number of $v'$ is
% &\le3a(j+1) \sum_{\Delta\ge0} \E T_g( m_g -b-\Delta+O(1) )\\
%& \times\max_{k_2\le g/2} \PP{L_{g,k_2}(2b+a+\Delta+2)|\tW_g=
%&\ll j \sum_{\Delta\ge0} n e^{-e(b+\Delta)} \cdot\frac{(a+b+
%&\ll j e^{-eb} (a+b+1)^6.
%Summing on $j$ and recalling the lower bound on $b$, we find that
Summing on $j$, we find that
\begin{equation}\label{smallj}
\sum_{0\le j\le r} F_j \ll1.
\end{equation}
Together, \eqref{biggestj}--\eqref{midj} and \eqref{smallj}
imply that for every $v\in T$,
\[
\E[ X \dvtx v\in\mathcal{X} ] = O(1).
\]
Combining this estimate with \eqref{eq:smm_set} and \eqref{EX} shows that
$\P\{ X>0 \} \gg1$, and if $X>0$, then $M_n \le m_n$, so there exists an
absolute constant
$\varepsilon> 0$ such that for all~$n$,
\[
\P\{ M_n \le m_n \} \ge\eps.
\]
From here it is straightforward to show that $M_n \le\tM_n + O(1)$,
and we now do so. The next two lemmas, taken from~\cite{ford09pratt},
are standard bounds for BRW. As the proofs are short, we include them here.

\begin{lemma}\label{tight1}
For any BRW, positive integers $m,n$ and positive real numbers~$M$, $N$,
\[
\P\{ M_{m+n} \ge M+N \} \le\E\bigl[ (\P\{ M_n\ge N \} )^{T_m(M)}\bigr].
\]
\end{lemma}

\begin{pf}
Suppose $M_{m+n} \ge M+N$ and $T_m(M) = k$. For each of these $k$
individuals, all of their descendants in
generation $m+n$ are offset from their generation~$m$ ancestor by at
least $N$.
\end{pf}

\begin{lemma}\label{tight2}
Let $m,n$ be positive integers and let $M>0, \eps>0$ be real.
If $\E\{ (1-\eps)^{T_m(M)} \} < \frac12$, then
$\P\{ M_n < \tM_{n+m} - M \} \le\eps.$
In particular, the conclusion holds if $\P\{ T_m(M) < 1/\eps\} \le
\frac15$.
\end{lemma}

\begin{pf}
Let $q=\sup\{x\dvtx\P\{M_n < x\} < \eps\}$; then $\P\{M_n < q\}
\leq
\eps$.
By Lem\-ma~\ref{tight1},
\[
\P\{ M_{m+n} \ge M+q \}
\le\E\bigl[ ( \P\{ M_n \ge q \} )^{T_m(M)} \bigr] < \tfrac12.
\]
Therefore, $M+q \ge\tM_{m+n}$, and, thus, $\P\{ M_n < \tM_{m+n} -
M\}
\le
\P\{ M_n < q \} \le\eps$. To prove the second part,
assume that $\P\{ T_m(M) < 1/\eps\} \le\frac15$.
Then
\begin{eqnarray*}
\E\bigl\{ (1-\eps)^{T_m(M)} \bigr\}& \le&\P
\biggl\{ T_m(m) < \frac{1}{\eps} \biggr\} + \biggl(1-\P\biggl\{T_m(M)<\frac{1}{\eps}\biggr\}\biggr)
(1-\eps)^{1/\eps} \\
&\le&\frac15 + \frac{4}{5\er} < \frac12.\hspace*{200pt}\qed
\end{eqnarray*}%
\noqed\end{pf}

Now take $A$ such that $\P\{ T_1(A)<1/\eps \} \le\frac15$. By Lemma
\ref{tight2},
\[
\P\{ M_n \le\tM_n-A \} \le\P\{ M_n\le\tM_{n+1}-A \} \le\eps
\]
and, hence, $\tM_{n} \le M_n + A$, which completes the proof of the
upper bound in Theorem~\ref{T1}.

%%%%%%%%%%%%%%%%%%%%%%%%%%%%%%%%%%%%%%%%%%%%%%%%%%%%%%%%%%%%%%%%%
%
%s7 ###
\section{\texorpdfstring{Proof of Theorem \protect\ref{TFord}}{Proof of Theorem 1.2}}\label{sec:tails}
%
%%%%%%%%%%%%%%%%%%%%%%%%%%%%%%%%%%%%%%%%%%%%%%%%%%%%%%%%%%%%%%%%%

Let $a>1/\er$ and $0<\eta<a\er/2$.
By Biggins' analog of Chernoff's inequality for the BRW \cite[Theorem
2]{Big77}, for large $r$ we have
$\P\{ T_r(ar) \le(a\er-\eta)^r \} \le\frac15$.
Let $r_0$ be large enough that, in addition, $\tM_{n+r}\ge\tM
_n+(1/\er
-\eta)r$
for all $r \geq r_0$ and all $n$ (such an $r_0$ exists by Theorem~\ref{T1}).
Now fix $r \geq r_0$ and let $M=ar$, let $m=r$, and let $\eps=(ae-\eta
)^{-r}$. We then have
$\P\{T_m(M)<1/\eps\} \le1/5$, so for all $n$, by the preceding bound
for $\tM_{n+r}$ and by Lemma~\ref{tight2}, we obtain that
\[
\P\{ M_n \le\tM_n-(a-1/\er+\eta)r \} \le
\P\{ M_n \le\tM_{n+r} - ar \} \le(a\er-\eta)^{-r}.
\]
The first estimate follows with $c_1=\frac{\log(a\er-\eta
)}{(a-1/\er
+\eta)}$.
Fix $a$, let $\eta\to0$, then let $a\to1/\er$,
so that $c_1\to\er$. This proves the first part of Theorem~\ref{TFord}.

For the second part, fix $0 < \eps< 1/50$ and let $\delta=\eps^2$,
so that
$\delta(1+\log((1-\eps/5)/\delta)) < \eps/5$.
Then choose $r_0$ sufficiently large that for all $r \geq r_0$, we have
$(1-\eps/5)r+2\lceil\log(2r) \rceil< r$, and for all $s \geq\log(2r_0)$,
we have $\P\{ T_s(2s) \le4^s \} \le e^{-1/\delta}$
(as in the first part, such an $r_0$ exists by \cite[Theorem 2]{Big77}).

Recall that if $h \in\mathbb{N}^1=T_1$ is a child of the root in $T$,
then $S(h)$ is
$\operatorname{Gamma}(h)$ distributed.
Thus, for any positive integer $r$, by a union bound
\begin{eqnarray*}
\P\bigl\{ T_1\bigl((1-\eps/5)r\bigr) \leq\delta r-1 \bigr\} &\leq&\sum_{h\le\delta r}
\P
\{
S(h) \geq(1-\eps/5)r\} \\
& =& \sum_{h\le\delta r} \frac{e^{-(1-\eps/5)r} ((1-\eps/5)
r)^h}{h!} \\
& \leq& e^{-(1-\eps/5)r} e^{(1+\log((1-\eps/5)/\delta))\delta r} \\
& \leq& e^{-(1-2\eps/5)r},
\end{eqnarray*}
the second-to-last inequality by Proposition~\ref{poisson}.

Write $s=\lceil\log(2r)\rceil$, and
let $E$ be the event that there are at least $4^s$ nodes in $T_{s+1}$
with displacement
at most $(1-\eps/5) r+2s < r$.
If $T_1((1-\eps/5) r) > \delta r$, then either $E$ occurs, or else for
each $h \leq
\lceil\delta r-1 \rceil$,
the number of $v \in T_{s+1}$ descending from $h\in T_1$
with $S(v)-S(h) \leq2s$ is less than $4^s$. The latter event has
probability less than
$(e^{-(1/\delta)})^{\delta r-1} = e^{(1/\delta)-r}$.
It follows that
\[
\P\{ E^c \} \leq e^{-(1-2\eps/5) r} + e^{(1/\delta)-r} \leq
e^{-(1-\eps/2)r},
\]
the last inequality holding for large $r$.
Finally, if $M_n > \tM_{n-(s+1)} + r$,
then for each node $v\in T_{s+1}$ with $S(v) \leq(1-\eps/5)r+2s$, for
all $w\in T_n$
descending from $v$ we must have $S(w)-S(v) \geq\tM_{n-(s+1)}$.
If $E$ occurs, then there are at least $4^s \geq2r$ such nodes $v$,
and so
\[
\P\bigl\{ M_n > \tM_{n-(s+1)} + r \bigr\} \leq e^{-(1-\eps/2)r}+2^{-2r} <
e^{-(1-\eps
) r},
\]
the last inequality holding for large $r$. Since $\tM_{n-(s+1)} \leq
\tM_n$,
the second part of Theorem~\ref{TFord} is proved by letting $\eps\to0$.

\section*{Acknowledgments}
The authors thank Hugh Montgomery for bringing paper~\cite{halmos_alms} to our attention.
K. Ford thanks the IAS for its hospitality and excellent working conditions.

% imsref loaded by akundreckaite, 2012-05-02 13:44:44
%

%suskaldyti doi

\printaddresses

\end{document}